\documentclass{amsart}
\usepackage{latexsym,amsmath,amssymb}
\newtheorem{theorem}{Theorem}[section]
\theoremstyle{remark}
\numberwithin{equation}{section}
\begin{document}

\title
{Finite sum of weighted composition operators with closed range}

\author{\sc S. Shamsigamchi}\author{\sc A. Alishahi} \author{\sc A. Ebadian}
\address{\sc Department of mathematics, Payame Noor University } \address{\sc Department of mathematics, Payame Noor University }
\email{saeedeh.shamsi@gmail.com} \email{}\email{ebadian.ali@gmail.com}

\thanks{}

\subjclass[2010]{47B33}

\keywords{Weighted composition operators, closed range operators, invertible operators.}

\date{}

\dedicatory{}

\commby{}

%%% ----------------------------------------------------------------------
\begin{abstract}
In this paper, first we characterize closedness of range of the finite sum of weighted composition operators between different $L^p$-spaces. Then we discuss polar decomposition and invertibility of these operators.
\end{abstract}

\maketitle
\section{introduction}
Weighted composition operators are a general class of operators and they appear naturally in the study of surjective isometries on most of the function spaces,
semigroup theory, dynamical systems, Brennans conjecture, etc. This type of operators are a generalization of multiplication operators and composition operators.\\
 There are many great papers on the investigation of weighted composition operators acting on the spaces of measurable functions. For instance, one can see \cite{bon,dls,ey,ej,sh,gh,j1,ls,lau6,n2}. Also, some basic properties of weighted composition operators on $L^{p}$-spaces were studied by Parrott \cite{la}, Nordgern \cite{la2}, Singh and Manhas \cite{lau}, Takagi \cite{n1} and some other mathematicians. As far as we know finite sum of weighted composition operators were studied on $L^p$-spaces by Jabbarzadeh and Estaremi in \cite{gh1}. Also we investigated some basic properties of these operators in \cite{bon}.

Let $(X, \Sigma, \mu)$ be a $\sigma$-finite measure space.
We denote the linear space of all complex-valued
$\Sigma$-measurable functions on $X$ by $L^0(\Sigma)$. For any $\sigma$-finite subalgebra $\mathcal{A}\subseteq\Sigma$ such that $(X, \mathcal{A}, \mu_{\mathcal{A}})$ is also $\sigma$-finite , the
conditional expectation operator associated with $\mathcal{A}$ is
the mapping $f\rightarrow E^{\mathcal{A}}f$, defined for all
non-negative $f$ as well as for all $f\in L^p(\Sigma)$, $1\leq
p\leq \infty$, where $E^{\mathcal{A}}f$ is the unique
$\mathcal{A}$-measurable function satisfying
$$\int_{A}fd\mu=\int_{A}E^{\mathcal{A}}fd\mu, \ \ \  A\in \mathcal{A}.$$

As an operator on $L^{p}({\Sigma})$, $E^{\mathcal{A}}$ is
idempotent and $E^{\mathcal{A}}(L^p(\Sigma))=L^p(\mathcal{A})$.
For more details on the properties of $E^{\mathcal{A}}$ see
\cite{lambe} and \cite{rao}.
For a measurable function $u:X\rightarrow \mathcal{C}$ and non-singular measurable transformation $\varphi:X\rightarrow X$,  i.e, the measure $\mu\circ \varphi^{-1}$ is absolutely continuous with respect to $\mu$,  we can define an operator $uC_{\varphi}:L^p(\Sigma)\rightarrow L^0(\Sigma)$ with $uC_{\varphi}(f)=u.f\circ \varphi$ and it is called a weighted composition operator. For non-singular measurable transformations $\{\varphi_i\}^{n}_{i=1}$, we put $W=\sum_{i=1}^{n}u_iC_{\varphi_i}$.\\
In this paper, we are going to give some sufficient and necessary condition for closedness of range of finite sum of weighted composition operators between different $L^p$-spaces. Moreover, we compute the polar decomposition of these operators on $L^2$. Finally we talk a bit about invertibility and injectivity.

\section{main results}
In this section first we give an equivalent condition  for closedness of range on the Hilbert space $L^2$.
\begin{theorem}
Let  $W=\sum_{i=1}^{n}u_{i}C_{\varphi_{i}}$ be a bounded operator on $L^{2}(\mu)$ and $u_{i}(\varphi_{j}^{-1})=0 ,~~~ i\neq j$. The following statements are equivalent.
\begin{enumerate}
\item[(a)] $W$ has closed range.
\item[(b)] There is a constant $c>0$ such that $J=\sum_{i=1}^{n}h_{i}E_{i}(|u_{i}|^{2})\circ \varphi_{i}^{-1} \geq c ~~~~~\mu - a.e$ on $Coz J=\{x\in X: J(x)\neq 0\}$.
\end{enumerate}
\end{theorem}
\begin{proof}
 $(b)\Rightarrow (a)$ Suppose that there is some constant $c>0$ such that $J\geq 0 ~~~~\mu-a.e$ on $Coz J$. We know that $\ker W \subseteq L^{2}_{|X\backslash Coz J}(\mu) $.  Since $W^{*}Wf=Jf$ for every $f\in L^{2}(\mu)$,
\begin{eqnarray}
\nonumber
\|Wf\|_{2}^{2} &=& (Wf , Wf)\\
 \nonumber & =& (W^{*}Wf , f)\\
\nonumber &=& \int_{X}J|f|^{2}d\mu = \int_{Coz J}J|f|^{2}d\mu+\int_{X\backslash Coz J}J|f|^{2}d\mu\\
\nonumber &\geq & c\|f\|_{2}^{2}.\\
\nonumber
\end{eqnarray}
Obviously $W_{| J}$ is injective and $W_{| J}(L^{2}_{\mid J}(\mu))$ is closed in $L^{2}(\mu)$, where $L^{2}_{\mid J}(\mu)=\{f\in L^{2}(\mu)~;~ f=0~ on~ X\backslash J\}$. Since $\ker W=L^{2}_{|X\backslash Coz J}(\mu) $, $W(L^{2}(\mu))$ must be closed in $L^{2}(\mu)$.\\
$(a)\Rightarrow (b)$ Assume $W$ has closed range. Then $W_{| J}(L^{2}_{| J}(\mu))$ is closed in $L^{2}(\mu)$. Since $W_{| J}$ is injective so there exists a constant $d>0$ such that $\|W_{| J}\|_{2}\geq d \|f\|_{2}$  for any $f\in L^{2}(\mu)$. Take $c=\frac{d^{2}}{n}$, $(b)$ follows immediately once we show that for any $E\in \Sigma$ with $E\subset Coz J$,  $\int_{E}J d\mu\geq c\mu(E)$. PicK any $E\in \Sigma$ with $E\subset J$. We may assume $\mu(E)<\infty$. Then $\chi_{E}\in L^{2}_{| J}(\mu)$ and $n\int_{E}J d\mu=n\int_{X}J \chi_{E}d\mu\geq \|W_{| J}\chi_{E}\|\geq d^{2}\|\chi_{E}\|_{2}^{2}=d^{2}\mu(E)$ so $\int_{E}J d\mu \geq c \mu(E)$.
\end{proof}
%%%%%%%%%%%%%%%%%%%%%%%%%%%%%%%%%%%%%%%%%%
Now we find some necessary and sufficient conditions for closedness of range when the operator act on the $L^p$ with $p>1$.

\begin{theorem}
Let  $W=\sum_{i=1}^{m}u_{i}C_{\varphi_{i}}$ be a bounded operator on $L^{p}(\mu)$ with $p>1$. Then the  followings hold.
\begin{enumerate}
\item[(a)]  If $J(B)=0, ~~\mu -a.e$  and $\sum_{i\in \Bbb N}J(A_{i})\mu(A_{i})<\infty$  then $W$ has closed range.
\item[(b)] If  $W$  has closed range and is injective then $J(B)=0, ~~\mu -a.e$ .
\item[(c)] Let $\mu(X)<\infty$. If  $W$  has closed range and is injective then  there exists a constant $\delta>0$ such that $u=\sum_{i=1}^{n}u_{i}^{p}\geq \delta$ ox $X$.
\end{enumerate}
\end{theorem}
\begin{proof}
 $(a)$ Take any sequence $(Wf_{n})_{n \in \Bbb N}$ in $W(L^{p}(\mu))$ with $\|f_{n}\|<1$ and $\|Wf_n-g\|\rightarrow 0$. For a fixed $i\in \Bbb N$ the sequence $(f_{n}(A_{i}))_{n \in \Bbb N}$ is bounded by $\frac{1}{\sqrt[p]{\mu(A_{i})}}$. So we can find a subsequence $(f_{n_{k}})_{k \in \Bbb N}$ such  that with each fixed $i$, $f_{n_{k}}(A_{i})\rightarrow \alpha_{i}$ for some  $\alpha_{i} \in \Bbb C$. Define $f= \sum_{i=1}^{\infty}\alpha_{i} \chi_{A_{i}}$. By Fatous lemma we have $\int_{X}|f|^{p}d\mu \leq \lim\inf_{k\rightarrow \infty}\int_{X}|f_{n_{k}}|^{p}d\mu \leq 1 $, or $f \in L^{p}(\mu)$. Then we have
\begin{eqnarray}
\nonumber \|g-Wf\|_{p}&\leq& \|g-Wf_{n}\|_{p}+\|Wf_{n}-Wf_{n_{k}}\|_{p}+\|Wf_{n_{k}}-Wf\|_{p}\\
 \nonumber &\leq& \frac{\epsilon}{3}+\frac{\epsilon}{3}+\int_{X}|W(f_{n_{k}}-f)|^{p}d\mu\\
  \nonumber &\leq& \frac{\epsilon}{3}+\frac{\epsilon}{3}+m^{p-1}\int_{X}J|f_{n_{k}}-f|^{p}d\mu\\
  \nonumber &= &\frac{\epsilon}{3}+\frac{\epsilon}{3}+m^{p-1}\sum_{i=1}^{\infty}J(A_{i})|f_{n_{k}}(A_{i})-\alpha_{i}|^{p}\mu(A_{i})\\
 \nonumber &\longrightarrow&  0\\
   \nonumber
\end{eqnarray}
 Obviously  $W(L^{p}(\mu))$ is closed in $L^{p}(\mu)$. \\
 $(b)$ Suppose on the contrary, $\mu(\{x\in B~:~J(x)>0\})>0$. Then there exists $\delta>0$ such that the set $G=\{x\in B~:~J(x)\geq\delta\}$ has positive measure. We assume $\mu(G)<\infty$. Moreover, as $G$ is non atomic, we can further assume that $\mu(X\backslash G)>0$. Consider the Banach space $L^{p}_{\mid_{G}}(\mu)$ and the operator $W_{\mid_{G}}$ defined on $L^{p}_{\mid_{G}}(\mu)$. We claim that $W_{\mid_{G}}(L^{p}_{\mid_{G}}(\mu))$ is closed in $L^{p}(\mu)$. To prove we take  any convergent sequence $(W_{\mid_{G}}(f_{n}))_{n\in \Bbb N}$ in $W\mid_{G}(L^{p}_{\mid_{G}}(\mu))$. Let $g\in L^{p}(\mu)$ satisfy
 $\|W_{\mid_{G}}(f_{n})-g\|_{p}\rightarrow 0$ as $n\rightarrow \infty$. Note that $(W_{\mid_{G}}(f_{n}))_{n\in \Bbb N}$ can be regarded as a sequence in $W(L^{p}(\mu))$. The closedness of range of $W$ yields an $f\in L^{p}(\mu)$ with $g=Wf~~ \mu - a.e$ On $X$. Then assume $W$ has closed range and is injective so there exists a constant $d>0$ such that $\|W_{\mid_{G}}(f_{n})-Wf\|_{p}\geq d\|f_{n}-f\|_{p}$. As $\|W_{\mid_{ G}}(f_{n})-g\|_{p}=\|W_{\mid_{G}}(f_{n})-Wf\|_{p}=0$ and $\|f_{n}-f\|_{p}^{p}=\int_{G}|f_{n}-f|^{p}d\mu+\int_{X\backslash G}|f_{n}-f|^{p}d\mu$ we have that $\int_{X\backslash G}|f|^{p}d\mu=0$ and so $f\in L^{p}_{\mid_{G}}(\mu)$. Then there exists some constant $c>0$ such that $\|W_{\mid_{G}}f\|_{p}\geq c\|f\|_{p}$ for all $f\in L^{p}_{\mid_{G}}(\mu)$. We claim that this is impossible by showing that for any $\alpha>0 $, there is some $f_{\alpha}\in L^{p}_{\mid_{G}}(\mu)$ satisfying $\|W_{\mid_{G}}f\|_{p}<c\|f\|_{p}$. For any $n\in \Bbb N$, define $G_{n}=\{x\in G~;~\left(\frac{(n-1)\alpha}{m^{\frac{p-1}{p}}}\right)^{p}\leq J(x)\leq \left(\frac{n\alpha}{m^{\frac{p-1}{p}}}\right)^{p}\}$. Then $G=\left(\cup_{n\in \Bbb N}G_{n}\right)\cup\{x\in G ~;~ J(x)=\infty\}$. Since $W$ is a bounded operator on $L^{p}(\mu)$ so $J$ is finite valued $\mu$-a.e on $X$, then we have $\mu(\{x\in G ~;~ J(x)=\infty\})=0$. Now as $\mu(G)>0$, $\mu(G_{N})>0$ for some $N\in \Bbb N$. Since $G_{N}$ is non- atomic, for any $\alpha>0$, we can choose some set $E_{\alpha}\in \Sigma$ such that $E_{\alpha}\subseteq G_{N}$ and $\mu(E_{\alpha})\leq\mu (G_{N})$. Take $f_{\alpha}=\chi_{E_{\alpha}}$. Obviously $f_{\alpha}\in L^{p}_{\mid_{G}}(\mu)$. Moreover $\|W_{\mid_{G}}f_{\alpha}\|_{p}\leq m^{\frac{p-1}{p}}\left(\int_{X}J|f_{\alpha}|^{p}d\mu\right)^{\frac{1}{p}}<m^{\frac{p-1}{p}}\left(\frac{N\alpha}{m^{\frac{p-1}{p}}}\right)\|f_{\alpha}\|_{p}=N\alpha \|f_{\alpha}\|_{p}$. This prove our  claim  and therefore we must have $J=0, ~~\mu -a.e$ on $B$.     \\
$(c)$ Assume $W$ has closed range and is injective so there exists a constant $d>0$ such that $\|Wf\|_{p}\geq d \|f\|_{p}$  for any $f\in L^{p}(\mu)$.
\begin{eqnarray}
\nonumber n^{p-1}\int_{X}\sum_{i=1}^{n}u_{i}^{p} d\mu&\geq & n^{p-1}\sum_{i=1}^{n}\int_{\varphi_{i}^{-1}(X)}u_{i}^{p}d\mu\\
\nonumber & = &n^{p-1}\sum_{i=1}^{n}\int_{X}u_{i}^{p}\chi_{\varphi_{i}^{-1}(X)}d\mu\\
\nonumber &\geq &\int_{X}|W\chi_{X}|^{p}d\mu\\
\nonumber &= &\|W\chi_{X}\|_{p}^{p}\\
 \nonumber &\geq & d^{p}\|\chi_{X}\|_{p}^{p}=d^{p}\mu(X)\\
\nonumber
\end{eqnarray}
 so $u \geq \delta$ on $X$ . The proof is now complete.
\end{proof}
%%%%%%%%%%%%%%%%%%%%%%%%%%%%%%%%%%%%%%%%%%%%%%%%%%%%
Here we give some necessary and sufficient conditions for closedness of range when the operator projects the $L^p$ into $L^{q}$ when $1\leq q<p<\infty$.

\begin{theorem}
Suppose that $1\leq q< p<\infty$ and let $W$ be a bounded operator from $L^{p}(\mu)$ into $L^{q}(\mu)$. The followings hold.
\begin{enumerate}
\item[(a)] If  $W$  has closed range and is injective then the set $\{i\in\Bbb N~;~J(=\sum_{r=1}^{n}h_{r}E_{r}(|u_{r}|^{q})\circ \varphi_{r}^{-1})(A_{i})>0\}$ is finite.
\item[(b)] If $J(B)=0$, $\mu-a.e$~~  and the set $\{i\in\Bbb N~;~J(=\sum_{r=1}^{n}h_{r}E_{r}(|u_{r}|^{q})\circ \varphi_{r}^{-1})(A_{i})>0\}$ is finite then $W$  has closed range.
\end{enumerate}
\end{theorem}
\begin{proof}
$(a)$ Suppose on the contrary, the set $\{i\in\Bbb N~;~J(A_{i})>0\}$ is infinite. Since $W$  is injective and has closed range there exists $d>0$ such that $\|Wf\|_{q}\geq d\|f\|_{p}$ for all $f\in L^{p}(\mu)$. Thus for any $i\in \Bbb N$, $\|W\chi_{A_{i}}\|_{q}^{q}\geq d^{q}\mu(A_{i})^{\frac{q}{p}}$ and so
\begin{eqnarray}
\nonumber d^{q}\mu(A_{i})^{\frac{q}{p}}&\leq & \|W\chi_{A_{i}}\|_{q}^{q}\\
\nonumber & = &\int_{X}|\sum_{r=1}^{n}u_{r}\chi_{A_{i}}\circ \varphi_{r}|^{q}d\mu\\
\nonumber &\leq & n^{q-1}\int_{X}J\chi_{A_{i}}d\mu\\
\nonumber &= &n^{q-1}J(A_{i})\mu(A_{i}).\\
 \nonumber
\end{eqnarray}
It follows from the preceding inequality that
\begin{eqnarray}
\nonumber \frac{d^{\frac{pq}{p-q}}}{n^{\frac{p(q-1)}{p-q}}}\leq J(A_{i})^{\frac{p}{p-q}}\mu(A_{i}).\\
 \nonumber
\end{eqnarray}
Therefore,
\begin{eqnarray}
\nonumber \infty=\sum_{i\in\Bbb N}\frac{d^{\frac{pq}{p-q}}}{n^{\frac{p(q-1)}{p-q}}}\leq \sum_{i\in\Bbb N}J(A_{i})^{\frac{p}{p-q}}\mu(A_{i})<\infty.\\
 \nonumber
\end{eqnarray}
This is a contradiction.\\
$(b)$ Let $g\in \overline{W(L^{p}(\mu))}$ then there exists a sequence $(Wf_{n})_{n\in\Bbb N}\subseteq W(L^{p}(\mu))$ such that $Wf_{n}\longrightarrow g$ and $\|f_{n}\|<1$. If the set $\{i\in\Bbb N~;~J(A_{i})>0\}$ is empty then $W$ is the zero operator . Otherwise we may assume there exists  some $k\in \Bbb N$ such that $J(A_{i})>0$ for $1\leq i \leq k$ and $J(A_{i})=0$ for any $i>k$. As $f_{n}\in L^{p}(\mu)$ for all $n$, $|f_{n}(A_{i})|\leq\frac{\|f_{n}\|_{p}}{\sqrt[p]{\mu(A_{i})}}\leq\frac{1}{\sqrt[p]{\mu(A_{i})}}$ for any $1\leq i\leq k$ and any $n\in \Bbb N$. By Bolzano-Weierstrass there exists a subsequence of nutural number $(n_{j})_{j\in \Bbb N}$ such that for each fixed $1\leq i\leq k$ the sequence $(f_{n_{j}}(A_{i}))_{j\in \Bbb N}$ converges. Suppose $\lim_{j\rightarrow \infty}f_{n_{j}}(A_{i})=\varsigma_{j}(\in \Bbb C)$ and define $f=\sum_{i=1}^{k}\varsigma_{j}\chi_{A_{i}}$. Then $f\in L^{p}(\mu)$. For every $\epsilon>0$  we have that
\begin{eqnarray}
\nonumber \|g-Wf\|_{q}&\leq& \|g-Wf_{n}\|_{q}+\|Wf_{n}-Wf_{n_{j}}\|_{q}+\|Wf_{n_{j}}-Wf\|_{q}\\
 \nonumber &\leq& \frac{\epsilon}{3}+\frac{\epsilon}{3}+\int_{X}|W(f_{n_{j}}-f)|^{q}d\mu\\
  \nonumber &\leq& \frac{\epsilon}{3}+\frac{\epsilon}{3}+n^{q-1}\int_{X}J|f_{n_{j}}-f|^{q}d\mu\\
  \nonumber &= &\frac{\epsilon}{3}+\frac{\epsilon}{3}+n^{q-1}\sum_{i=1}^{k}J(A_{i})|f_{n_{j}}(A_{i})-\varsigma_{j}|^{q}\mu(A_{i})\\
 \nonumber &\longrightarrow&  0\\
   \nonumber
\end{eqnarray}
\end{proof}
%%%%%%%%%%%%%%%%%%%%%%%%%%%%%%%%%%%%%%%%%%%%%%%%%%%%
In the next theorem we obtain some necessary and sufficient conditions for closedness of range when the operator projects the $L^p$ into $L^{q}$ when $1\leq p<q<\infty$.
\begin{theorem}
Suppose that $1\leq p< q<\infty$ and  $W=\sum_{i=1}^{m}u_{i}C_{\varphi_{i}}$ be a bounded operator from $L^{p}(\mu)$ into $L^{q}(\mu)$.   Then the  followings hold.
\begin{enumerate}
\item[(a)]  If $J(B)=0, ~~\mu -a.e$  and $\sum_{i\in \Bbb N}J(A_{i})\mu(A_{i})<\infty$  then $W$ has closed range.
\item[(b)] If  $W$  has closed range and is injective then $J(B)=0, ~~\mu -a.e$ .
\item[(c)] Let $\mu(X)<\infty$. If  $W$  has closed range and is injective then  there exists a constant $\delta>0$ such that $u=\sum_{i=1}^{n}u_{i}^{p}\geq \delta$ on $X$.
\end{enumerate}
\end{theorem}
\begin{proof}
 $(a)$ Take any sequence $(Wf_{n})_{n \in \Bbb N}$ in $W(L^{p}(\mu))$ with $\|f_{n}\|<1$. For fixed $i\in \Bbb N$ the sequence $(f_{n}(A_{i}))_{n \in \Bbb N}$ is bounded by $\frac{1}{\sqrt[p]{\mu(A_{i})}}$. Applying contor's diagonalization procces, we extract a subsequence $(f_{n_{k}})_{k \in \Bbb N}$ such  that with each fixed $i$, $f_{n_{k}}(A_{i})\rightarrow \alpha_{i}$ for each  $\alpha_{i} \in \Bbb C$. Define $f= \sum_{i=1}^{\infty}\alpha_{i} \chi_{A_{i}}$. By fatous lemma we have $\int_{X}|f|^{p}d\mu \leq \lim\inf_{k\rightarrow \infty}\int_{X}|f_{n_{k}}|^{p}d\mu \leq 1 $, or $f \in L^{p}(\mu)$. Then we have
\begin{eqnarray}
\nonumber \|g-Wf\|_{q}&\leq& \|g-Wf_{n}\|_{q}+\|Wf_{n}-Wf_{n_{k}}\|_{q}+\|Wf_{n_{k}}-Wf\|_{q}\\
 \nonumber &\leq& \frac{\epsilon}{3}+\frac{\epsilon}{3}+\int_{X}|W(f_{n_{k}}-f)|^{q}d\mu\\
  \nonumber &\leq& \frac{\epsilon}{3}+\frac{\epsilon}{3}+m^{q-1}\int_{X}J|f_{n_{k}}-f|^{q}d\mu\\
  \nonumber &= &\frac{\epsilon}{3}+\frac{\epsilon}{3}+m^{q-1}\sum_{i=1}^{\infty}J(A_{i})|f_{n_{k}}(A_{i})-\alpha_{i}|^{q}\mu(A_{i})\\
 \nonumber &\longrightarrow&  0\\
   \nonumber
\end{eqnarray}
 Obviusly  $W(L^{p}(\mu))$ is closed in $L^{q}(\mu)$. \\
 $(b)$ Suppose on the contrary, $\mu(\{x\in B~:~J(x)>0\})>0$. Then there exists some $\delta>0$ such that the set $G=\{x\in B~:~J(x)\geq\delta\}$ has positive $\mu$- measure. We assume $\mu(G)<\infty$. Moreover, as $G$ is non atomic, we can further assume that $\mu(X\backslash G)>0$. Consider the Banach space $L^{p}_{\mid_{G}}(\mu)$ and the operator $W_{\mid_{G}}$ defined on $L^{p}_{\mid_{G}}(\mu)$. We claim that $W_{\mid_{G}}(L^{p}_{\mid_{G}}(\mu))$ is closed in $L^{q}(\mu)$. To prove we take  any convergent sequence $(W_{\mid_{G}}(f_{n}))_{n\in \Bbb N}$ in $W_{\mid_{G}}(L^{p}_{\mid_{G}}(\mu))$. Let $g\in L^{q}(\mu)$ satisfy
 $\|W_{\mid_{G}}(f_{n})-g\|_{q}\rightarrow 0$ as $n\rightarrow \infty$. Note that $(W_{\mid_{G}}(f_{n}))_{n\in \Bbb N}$ can be ragarded as a sequence in $W(L^{p}(\mu))$. The closedness of range of $W$ yeilds an $f\in L^{p}(\mu)$ with $g=Wf~~ \mu - a.e$ On $X$. Then assume $W$ has closed range and is injective so there exists a constant $d>0$ such that $\|W_{\mid_{G}}(f_{n})-Wf\|_{q}\geq d\|f_{n}-f\|_{p}$. As $\|W_{\mid_{ G}}(f_{n})-g\|_{q}=\|W_{\mid_{G}}(f_{n})-Wf\|_{q}=0$ and $\|f_{n}-f\|_{p}^{p}=\int_{G}|f_{n}-f|^{p}d\mu+\int_{X\backslash G}|f_{n}-f|^{p}d\mu$ we have that $\int_{X\backslash G}|f|^{p}d\mu=0$ and so $f\in L^{p}_{\mid_{G}}(\mu)$. Then there exists some conctant $c>0$ such that $\|W_{\mid_{G}}f\|_{q}\geq c\|f\|_{p}$ for all $f\in L^{p}_{\mid_{G}}(\mu)$. We claim that this is impossible by showing that for any $\alpha>0 $, there is some $f_{\alpha}\in L^{p}_{\mid_{G}}(\mu)$ satisfying $\|W_{\mid_{G}}f\|_{q}<c\|f\|_{p}$. For any $n\in \Bbb N$, define $G_{n}=\{x\in G~;~n-1\leq J(x)\leq n\}$. Then $G=\left(\cup_{n\in \Bbb N}G_{n}\right)\cup\{x\in G ~;~ J(x)=\infty\}$. Since $W$ is a bounded operator on $L^{p}(\mu)$ so $J$ is finite valued $\mu$-a.e on $X$, then we have $\mu(\{x\in G ~;~ J(x)=\infty\})=0$. Now as $\mu(G)>0$, $\mu(G_{N})>0$ for some $N\in \Bbb N$. Since $G_{N}$ is non- atomic, for any $\alpha>0$, we can choose some set $E_{\alpha}\in \Sigma$ such that $E_{\alpha}\subseteq G_{N}$ and $\mu(E_{\alpha})=\frac{\mu (G_{N})}{K}$, where $K< \frac{N^{\frac{q}{q-p}}\mu(G_{N}}{\alpha^{\frac{pq}{q-p}}}$. Take $f_{\alpha}=\chi_{E_{\alpha}}$. Obviously $f_{\alpha}\in L^{p}_{\mid_{G}}(\mu)$. Moreover
 \begin{eqnarray}
\nonumber
 \|W_{\mid_{G}}f_{\alpha}\|_{p}&\leq & m^{\frac{q-1}{q}}\left(\int_{X}J|f_{\alpha}|^{q}d\mu\right)^{\frac{1}{q}}\\
\nonumber & < & m^{\frac{q-1}{q}}\left(\frac{N\mu(G_{N}}{K}\right)^{\frac{1}{q}}\\
\nonumber &=& m^{\frac{q-1}{q}}N^{\frac{1}{q}}\left(\frac{\mu(G_{N}}{K}\right)^{\frac{1}{p}+\frac{p-q}{pq}}\\
\nonumber &=& m^{\frac{q-1}{q}}N^{\frac{1}{q}}\|f_{\alpha}\|_{p}\left(\frac{K}{\mu(G_{N}}\right)^{\frac{q-p}{pq}}\\
\nonumber &<& \frac{N}{\alpha}\|f_{\alpha}\|_{p}.
\nonumber
\end{eqnarray}
 This prove our  claim  and therefore we must have $J=0, ~~\mu -a.e$ on $B$.     \\
$(c)$ Assume $W$ has closed range and is injective so there exists a constant $d>0$ such that $\|Wf\|_{p}\geq d \|f\|_{p}$  for any $f\in L^{p}(\mu)$. %For any $E\in \Sigma$,  we may assume $\mu(E)<\infty$. Take $\delta=\frac{d^{p}}{n^{p-1}\mu(X)}$,   Then, % $\chi_{E}\in L^{p}(\mu)$ and
\begin{eqnarray}
\nonumber n^{p-1}\int_{X}\sum_{i=1}^{n}u_{i}^{p} d\mu&\geq & n^{p-1}\sum_{i=1}^{n}\int_{\varphi_{i}^{-1}(X)}u_{i}^{p}d\mu\\
\nonumber & = &n^{p-1}\sum_{i=1}^{n}\int_{X}u_{i}^{p}\chi_{\varphi_{i}^{-1}(X)}d\mu\\
\nonumber &\geq &\int_{X}|W\chi_{X}|^{p}d\mu\\
\nonumber &= &\|W\chi_{X}\|_{p}^{p}\\
 \nonumber &\geq & d^{p}\|\chi_{X}\|_{p}^{p}=d^{p}\mu(X)\\
\nonumber
\end{eqnarray}
 so $u \geq \delta$ on $X$ . The proof is now complete.
\end{proof}

%%%%%%%%%%%%%%%%%%%%%%%%%%%%%%%%%%%%%%%%%%%%%%%%%%%%%%%%%%%%%

%%%%%%%%%%%%%%%%%%%%%%%%%%%%%%%%%%%%%%%%%%%%%%%%%%%%%%%%%%%%%%%%%
In the sequel we investigate the closedness of range the operator in from $L^{\infty}$ into $L^q$ and the converse. First we find some necessary and sufficient conditions for the case that $W$ is a bounded operator from $L^{\infty}$ into $L^q$ with $1<q<\infty$.
\begin{theorem}
Suppose that $1\leq q<\infty$. Let $J=\sum_{r=1}^{n}h_{r}E_{r}(|u_{r}|^{q})\circ \varphi_{r}^{-1}$ and $W$ be a operator from $L^{\infty}(\mu)$ into $L^{q}(\mu)$. The followings hold.
\begin{enumerate}
\item[(a)]
If
\begin{enumerate}
\item[(1)] $W$  has closed range.
\item[(2)]  $W$ is injective.
\item[(3)]  $\sum_{i\in\Bbb N}J(A_{i})\mu(A_{i})<\infty.$
\end{enumerate}

  Then the set $\{i\in\Bbb N~;~J(A_{i})>0\}$ is finite.
\item[(b)] If $J(B)=0$, $\mu-a.e$~~  and the set $\{i\in\Bbb N~;~J(A_{i})>0\}$ is finite then $W$  has closed range.
\end{enumerate}
\end{theorem}
\begin{proof}
$(a)$ Suppose on the contray, the set $\{i\in\Bbb N~;~J(A_{i})>0\}$ is infinite. Since $W$  has closed range and is injective we can find some constant$d>0$ such that $\|Wf\|_{q}\geq d\|f\|_{\infty}$ for all $f\in L^{\infty}(\mu)$. Thus for any $i\in \Bbb N$, $\|W\chi_{A_{i}}\|_{q}^{q}\geq d^{q}$ and so we have
\begin{eqnarray}
\nonumber d^{q}&\leq & \|W\chi_{A_{i}}\|_{q}^{q}\\
\nonumber & = &\int_{X}|\sum_{r=1}^{n}u_{r}\chi_{A_{i}}\circ \varphi_{r}|^{q}d\mu\\
\nonumber &\leq & n^{q-1}\int_{X}J\chi_{A_{i}}d\mu\\
\nonumber &= &n^{q-1}J(A_{i})\mu(A_{i})\\
 \nonumber
\end{eqnarray}
It follows from the preceding inequality that
\begin{eqnarray}
\nonumber \frac{d^{q}}{n^{q-1}}\leq J(A_{i})\mu(A_{i})\\
 \nonumber
\end{eqnarray}
Therefore,
\begin{eqnarray}
\nonumber \infty=\sum_{i\in\Bbb N}\frac{d^{q}}{n^{q-1}}\leq \sum_{i\in\Bbb N}J(A_{i})\mu(A_{i})<\infty\\
 \nonumber
\end{eqnarray}
contradiction arises. \\
$(b)$ Let $g\in \overline{W(L^{\infty}(\mu))}$ then there exists a sequence $(Wf_{n})_{n\in\Bbb N}\subseteq W(L^{\infty}(\mu))$ such that$Wf_{n}\longrightarrow g$ with $\|f_{n}\|<1$. If the set $\{i\in\Bbb N~;~J(A_{i})>0\}$ is empty then $W$ is the zero operator . Otherwise we may assume there exists  some $k\in \Bbb N$ such that $J(A_{i})>0$ for $1\leq i \leq k$ and $J(A_{i})=0$ for any $i>k$. As $f_{n}\in L^{\infty}(\mu)$ for all $n$, $|f_{n}(A_{i})|\leq\|f_{n}\|_{\infty}$ for any $1\leq i\leq k$ and any $n\in \Bbb N$. By Bolzano-Weierstrass there exists a subsequence of nutural number $(n_{j})_{j\in \Bbb N}$ such thst for each fixed $1\leq i\leq k$ the sequence $(f_{n_{j}}(A_{i}))_{j\in \Bbb N}$ converjes. Suppose $\lim_{j\rightarrow \infty}f_{n_{j}}(A_{i})=\varsigma_{j}(\in \Bbb C)$ and define $f=\sum_{i=1}^{k}\varsigma_{j}\chi_{A_{i}}$. Then $f\in L^{\infty}(\mu)$. For every $\epsilon>0$  we have that
\begin{eqnarray}
\nonumber \|g-Wf\|_{q}&\leq& \|g-Wf_{n}\|_{q}+\|Wf_{n}-Wf_{n_{j}}\|_{q}+\|Wf_{n_{j}}-Wf\|_{q}\\
 \nonumber &\leq& \frac{\epsilon}{3}+\frac{\epsilon}{3}+\int_{X}|W(f_{n_{j}}-f)|^{q}d\mu\\
  \nonumber &\leq& \frac{\epsilon}{3}+\frac{\epsilon}{3}+n^{q-1}\int_{X}J|f_{n_{j}}-f|^{q}d\mu\\
  \nonumber &= &\frac{\epsilon}{3}+\frac{\epsilon}{3}+n^{q-1}\sum_{i=1}^{k}J(A_{i})|f_{n_{j}}(A_{i})-\varsigma_{j}|^{q}\mu(A_{i})\\
 \nonumber &\longrightarrow&  0\\
   \nonumber
\end{eqnarray}
\end{proof}
%%%%%%%%%%%%%%%%%%%%%%%%%%%%%%%%%%%%%%%%%%%%%%%%%%%%%%%%%%%%%%%
Now we find some necessary and sufficient conditions for the case that $W$ is a bounded operator from $L^p$ into $L^{\infty}$ with $1<p<\infty$.
\begin{theorem}
Let   $u_{i}$'s are nonnegative and $\mu(X)<\infty$. Suppose that $1\leq p<\infty$ and let $W$ be a operator from $L^{p}(\mu)$ into $L^{\infty}(\mu)$. The followings hold.
\begin{enumerate}
\item[(a)]  If $(X, \Sigma, \mu)$ be a purely atomic space and $W$ is  bounded operator   then $W$ has closed range.
\item[(b)] If  $W$  has closed range and is injective then  there exists a constant $\delta>0$ such that $u=\sum_{i=1}^{n}u_{i}^{p}\geq \delta$ on $X$.
\end{enumerate}
\end{theorem}
\begin{proof}
 $(a)$ Take any sequence $(Wf_{n})_{n \in \Bbb N}$ in $W(L^{p}(\mu))$ with $\|f_{n}\|<1$. For fixed $i\in \Bbb N$ the sequence $(f_{n}(A_{i}))_{n \in \Bbb N}$ is bounded by $\frac{1}{\sqrt[p]{\mu(A_{i})}}$. Applying contor's diagonalization procces, we extract a subsequence $(f_{n_{k}})_{k \in \Bbb N}$ such  that with each fixed $i$, $f_{n_{k}}(A_{i})\rightarrow \alpha_{i}$ for each  $\alpha_{i} \in \Bbb C$. Define $f= \sum_{i=1}^{\infty}\alpha_{i} \chi_{A_{i}}$. By fatous lemma we have $\int_{X}|f|^{p}d\mu \leq \lim\inf_{k\rightarrow \infty}\int_{X}|f_{n_{k}}|^{p}d\mu \leq 1 $, or $f \in L^{p}(\mu)$. Then we have
\begin{eqnarray}
\nonumber \|g-Wf\|_{\infty}&\leq& \|g-Wf_{n}\|_{\infty}+\|Wf_{n}-Wf_{n_{k}}\|_{\infty}+\|Wf_{n_{k}}-Wf\|_{\infty}\\
 \nonumber &\leq& \frac{\epsilon}{3}+\frac{\epsilon}{3}+\|W\|\int_{X}|f_{n_{k}}-f|^{q}d\mu\\
  \nonumber &\leq& \frac{\epsilon}{3}+\frac{\epsilon}{3}+\|W\|\int_{\cup_{i\in \Bbb N}A_{i}}|f_{n_{k}}-f|^{q}d\mu\\
  \nonumber &= &\frac{\epsilon}{3}+\frac{\epsilon}{3}+\|W\|\sum_{i=1}^{\infty}|f_{n_{k}}(A_{i})-\alpha_{i}|^{q}\mu(A_{i})\\
 \nonumber &\longrightarrow&  0\\
   \nonumber
\end{eqnarray}
 Obviusly  $W(L^{p}(\mu))$ is closed in $L^{\infty}(\mu)$. \\
 $(b)$ Assume $W$ has closed range and is injective so there exists a constant $d>0$ such that $\|Wf\|_{\infty}\geq d \|f\|_{p}$  for any $f\in L^{p}(\mu)$. Take $\delta=\frac{d^{p}\mu(X)}{n^{p-1}}$, Then,
\begin{eqnarray}
\nonumber |\sum_{i=1}^{n}u_{i}\chi_{X}\circ \varphi_{i}|^{p} &\leq & n^{p-1}\sum_{i=1}^{n}u_{i}^{p}\\
\nonumber
\end{eqnarray}
Therefore,
\begin{eqnarray}
\nonumber n^{p-1}\sum_{i=1}^{n}u_{i}^{p}&\geq& (\sum_{i=1}^{n}u_{i})^{p}\\
%\nonumber &\geq& (\sum_{i=1}^{n}u_{i}\chi_{X}\circ \varphi_{i})^{p}\\
 \nonumber &\geq&\|W\chi_{X}\|_{\infty}^{p} \\
 \nonumber &\geq & d^{p}\|\chi_{X}\|_{p}^{p}=d^{p}\mu(X)\\
\nonumber
\end{eqnarray}
 so $u \geq \delta$ on $X$ . The proof is now complete.
\end{proof}
%%%%%%%%%%%%%%%%%%%%%%%%%%%%%%%%%%%%%%%%%%%%%%%%%%%%%%%
Here we consider $W$ as a bounded operator on $L^{\infty}$.
\begin{theorem}
Let   $u_{i}$'s are nonnegative and $\mu(X)<\infty$. Suppose that  $W$ be a bounded operator from $L^{\infty}(\mu)$ into $L^{\infty}(\mu)$. The followings hold.
\begin{enumerate}
\item[(a)]  If If $(X, \Sigma, \mu)$ be a purely atomic space then $W$ has closed range.
\item[(b)] If  $W$  has closed range and is injective then  there exists a constant $\delta>0$ such that $u=\sum_{i=1}^{n}u_{i}\geq \delta$ on $X$.
\end{enumerate}
\end{theorem}
\begin{proof}
 $(a)$ Take any sequence $(Wf_{n})_{n \in \Bbb N}$ in $W(L^{\infty}(\mu))$ with $\|f_{n}\|<1$. For fixed $i\in \Bbb N$ the sequence $(f_{n}(A_{i}))_{n \in \Bbb N}$ is bounded by $|f_{n}(A_{i})|\leq \|f_{n}\|<1$. Applying contor's diagonalization procces, we extract a subsequence $(f_{n_{k}})_{k \in \Bbb N}$ such  that with each fixed $i$, $f_{n_{k}}(A_{i})\rightarrow \alpha_{i}$ for each  $\alpha_{i} \in \Bbb C$. Define $f= \sum_{i=1}^{\infty}\alpha_{i} \chi_{A_{i}}$.  Then we have
\begin{eqnarray}
\nonumber \|g-Wf\|_{\infty}&\leq& \|g-Wf_{n}\|_{\infty}+\|Wf_{n}-Wf_{n_{k}}\|_{\infty}+\|Wf_{n_{k}}-Wf\|_{\infty}\\
 \nonumber &\leq& \frac{\epsilon}{3}+\frac{\epsilon}{3}+\|W\|\|f_{n_{k}}-f\|_{\infty}\\
  \nonumber &\leq& \frac{\epsilon}{3}+\frac{\epsilon}{3}+\|W\|\sup_{i\in \Bbb N}|f_{n_{k}}(A_{i})-\alpha_{i}|\\
   \nonumber &\longrightarrow&  0\\
   \nonumber
\end{eqnarray}
 Obviously  $W(L^{\infty}(\mu))$ is closed in $L^{\infty}(\mu)$. \\

$(b)$ Assume $W$ has closed range and is injective so there exists a constant $d>0$ such that $\|Wf\|_{\infty}\geq d \|f\|_{\infty}$  for any $f\in L^{\infty}(\mu)$.  Take $\delta=d$.   Then ,
\begin{eqnarray}
\nonumber |\sum_{i=1}^{n}u_{i}\chi_{X}\circ \varphi_{i}| &\leq & \sum_{i=1}^{n}u_{i}\\
\nonumber
\end{eqnarray}
Therefore,
\begin{eqnarray}
\nonumber \sum_{i=1}^{n}u_{i} &\geq& \sum_{i=1}^{n}u_{i}\chi_{X}\circ \varphi_{i}\\
%\nonumber &\geq& |W\chi_{X}\|_{\infty}\\
 \nonumber &\geq & d\\
\nonumber
\end{eqnarray}
 so $u \geq \delta$ on $X$ . The proof is now complete.
\end{proof}

%%%%%%%%%%%%%%%%%%%%%%%%%%%%
In the next theorem we obtain the polar decomposition of $W$ as a bounded operator on the Hilbert space $L^2$.
\begin{theorem}
 Suppose $u_{i}(\varphi_{j}^{-1})= 0,  ~~~~i\neq j$. The unique polar decomposition of $W=\sum_{i=1}^{n}u_{i}C_{\varphi_{i}}$ is $V|W|$ where $|W|(f)=M_{J}f$, $V(g)=\sum_{i=1}^{n}u_{i}\frac{\chi_{B}g}{\sqrt{J}}\circ \varphi_{i}$ and $B=Coz (J=\sum_{i=1}^{n}h_{i}E_{i}(u_{i}^{2})\circ \varphi_{i}^{-1})$.
\end{theorem}
\begin{proof}
We have that
\begin{eqnarray}
\nonumber
\|Wf\|_{2}^{2} &=& (Wf , Wf)\\
 \nonumber & =& (W^{*}Wf , f)\\
\nonumber &=& \int_{X}J|f|^{2}d\mu = \int_{B}J|f|^{2}d\mu+\int_{X\backslash B}J|f|^{2}d\mu\\
\nonumber &=& \int_{B}J|f|^{2}d\mu
\nonumber
\end{eqnarray}
where $J=\sum_{i=1}^{n}h_{i}E_{i}(u_{i})^{2}\circ \varphi_{i}^{-1}$. Then $\ker W=L^{2}(X\backslash B)=(L^{2}(B))^{\bot}$. For  each $f\in L^{2}(\mu)$ write $f=\chi_{B}f+\chi_{X\backslash B}f$ so that $Wf=W\chi_{B}f$. We may define partial isometry $V$ with initial space $(\ker W)^{\bot}=L^{2}(B)$ and final space $Ran W$ by $V(g)=\sum_{i=1}^{n}u_{i}\frac{\chi_{B}g}{\sqrt{J}}\circ \varphi_{i}, ~~~~~g\in L^{2}(\mu)$. Then the unique polar for $W$ is given $W=VM_{\sqrt{J}}$.
\end{proof}
Finally, the next two assertions we investigate invertibility of $W$.
\begin{theorem}
Let $(X, \Sigma, \mu)$ be apurely atomic measure space, $0\neq u_{i}\in L^{\infty}(\mu)$ and $W$ be a sum finite of weighted composition operators on $L^{p}(\mu)$. If there is a positive integer $N_{i}$ such that $\varphi_{i}^{N_{i}}(A_{n})=A_{n}$ up to a null set for all $n\geq 1$ and $u_{i}(\varphi_{j})=0,~~~i\neq j$ then
\begin{enumerate}
\item[(a)] $W$ is invertible.
\item[(b)] The set function $E$ that is defined as $E(B)=M_{\chi_{B}\circ v}$ for all borel sets $B$ of $\Bbb C$ is a  spectral measure where $v=\sum_{i=1}^{n}u_{i}u_{i}\circ \varphi_{i}\cdots u_{i}\circ \varphi_{i}^{N-1}$, $N=[N_{1}, \cdots, N_{n}]$.
 \end{enumerate}
\end{theorem}
\begin{proof}
$(a)$ Not that $\ker W^{r}\subseteq \ker W^{r+1}$ and $W^{r+1}(L^{p}(\mu))\supseteq W^{r}(L^{p}(\mu))$. If there is a positive integer $N_{i}$ such that $\varphi_{i}^{N_{i}}(A_{n})=A_{n}$ up to a null set for all $n\geq 1$ then $W^{N}$ is a multiplication operator induced by function $v=\sum_{i=1}^{n}u_{i}u_{i}\circ \varphi_{i}\cdots u_{i}\circ \varphi_{i}^{N-1}$,where  $N=[N_{1}, \cdots, N_{n}]$.\\
If $f\in \ker W^{N}$ then $W^{N}f(A_{n})=0$ for all $n\geq 1$. We have that $vf(A_{n})=0$ therefor $f=0$, $\mu-a.e$ on $X$. So $W$ is injective.\\
Let $g\in L^{p}(\mu)$ then $W^{N}(\frac{g}{v})(A_{n})=g(A_{n})$ for all $n\geq 1$. So  $W$ is surjective.\\
$(b)$ As observed by Rho and Yoo (\cite{Rho}, example 1), the multiplication operator $M_{\chi_{v}}$ is spectral. In fact the spectral measure $E$ is given by $E(B)=M_{\chi_{B}\circ v }$ for all Borel set $B$ of $\Bbb C$.
\end{proof}

\begin{theorem}
Let  $W=\sum_{i=1}^{n}u_{i}C_{\varphi_{i}}$ be a bounded operator on $L^{2}(\mu)$ and $u_{i}(\varphi_{j}^{-1})=0 ,~~~ i\neq j$. The following statements are equivalent.
\begin{enumerate}
\item[(a)] $W$ is  injective.
\item[(b)]  $J=\sum_{i=1}^{n}h_{i}E_{i}(|u_{i}|^{2})\circ \varphi_{i}^{-1} >0 ~~~~~\mu - a.e$ on $X$.
\item[(c)] whenever $J(E)=0$ for $E\in \Sigma$, $\mu(E)=0$.
\end{enumerate}
\end{theorem}
\begin{proof}
 $(b)\Rightarrow (a)$ Take any $f\in\ker W$, then we have
\begin{eqnarray}
\nonumber
0=\|Wf\|_{2}^{2} &=& (Wf , Wf)\\
 \nonumber & =& (W^{*}Wf , f)\\
\nonumber &=& \int_{X}J|f|^{2}d\mu = \int_{Coz f}J|f|^{2}d\mu+\int_{X\backslash Coz f}J|f|^{2}d\mu\\
\nonumber &=& \int_{Coz f}J|f|^{2}d\mu\\
\nonumber
\end{eqnarray}
Since  $J >0 ~~~~~\mu - a.e$ on $Coz f$, it follows that $\mu(Coz f)=0$ or $f=0~~~\mu - a.e$ on $X$.\\
$(a)\Rightarrow (c)$ Let  $E\in \Sigma$ satisfy $J(E)=0$ we may also assume $\mu(E)<\infty$. Then $\chi_{E} \in L^{2}(\mu)$ and $\|W\chi_{E}\|_{2}^{2}=\int_{X}J\chi_{E}d\mu=\int_{E}Jd\mu=0$. Now the injectivity of $W$ implies that $\chi_{E}=0, ~~~\mu - a.e$ on $X$. Hence $\mu(E)=0$. \\
$(c)\Rightarrow (b)$ Put $B=Coz J$. Clearly, $X\backslash B\in \Sigma$. Moreover, since $J(X\backslash B)=0$
We must have $\mu(X\backslash B)=0$. This shows that  $J >0, ~~~~~\mu - a.e$ on $X$.

\end{proof}

\end{document}